\documentclass[preprint]{elsarticle}
\usepackage{bbm}
\usepackage[all,cmtip]{xy}
\usepackage{amsfonts,amssymb,amsmath,amscd,amsthm,mathtools}
\usepackage{mathrsfs}
\usepackage{latexsym}
\usepackage{color}
\usepackage{verbatim,enumitem}
\usepackage{hyperref}
\input diagxy
\usepackage{tikz-cd}
\usepackage[total={155mm,225mm}]{geometry}
\usepackage{newtxmath}

\theoremstyle{plain}
\newtheorem{thm}{Theorem}[section]
\newtheorem{lem}[thm]{Lemma}
\newtheorem{prop}[thm]{Proposition}
\newtheorem{cor}[thm]{Corollary}
\theoremstyle{definition}
\newtheorem{defn}[thm]{Definition}

\newtheorem{exmp}[thm]{Example}
\newtheorem{rem}[thm]{Remark}

\newtheorem*{SA}{Standing Assumption}

\newcommand{\ra}{\rightarrow}
\newcommand{\lra}{\longrightarrow}
\newcommand{\id}{\mathrm{id}}
\newcommand{\with}{\&}

\begin{document}

\newcommand{\bfA}{\mathbf{A}}
\newcommand{\bfB}{\mathbf{B}}
\newcommand{\bfC}{\mathbf{C}}
\newcommand{\sQ}{\mathsf{Q}}
\newcommand{\sK}{\mathsf{K}}
\newcommand{\sLu}{\mathsf{\text{\L}}}
\newcommand{\sM}{\mathsf{M}}
\newcommand{\ev}{\mathrm{ev}}
\newcommand{\Set}{\mathbf{Set}}
\newcommand{\Cat}{\mathbf{Cat}}
\newcommand{\QCat}{Q\text{-}\Cat}
\newcommand{\RCat}{[0,1]\text{-}{\Cat}}
\newcommand{\KCat}{{K}\text{-}\Cat}
\newcommand{\qref}{\sQ\text{-}\mathbf{Ref}}
\newcommand{\QRel}{\sQ\text{-}\mathbf{Rel}}
\newcommand{\bbt}{\vmathbb{2}}
\newcommand{\Ra}{\Rightarrow}
\newcommand{\sft}{\mathsf{2}}
\newcommand{\Ord}{\mathbf{POrd}}
\newcommand{\bv}{\bigvee}
\newcommand{\bw}{\bigwedge}
\newcommand{\Lra}{\Longrightarrow}
\newcommand{\MCat}{{M}\text{-}\mathbf{Cat}}
\newcommand{\sfM}{\mathsf{M}}
\newcommand{\Idm}{\mathbf{Idm_{\with}}}
\newcommand{\Lu}{\text{\L}}
\newcommand{\YCom}{\mathbf{YCCat}}
\newcommand{\YComM}{{M}\text{-}\mathbf{YCCat} }
\newcommand{\YComK}{K\text{-}\mathbf{YCCat}}
\newcommand{\lam}{\lambda}
\newcommand{\de}{\delta}
\newcommand{\ga}{\gamma} 
\newcommand{\al}{\alpha}
\newcommand{\be}{\beta}
\newcommand{\xlam}{\{x_\lambda\}_{\lambda\in D}}
\newcommand{\zlam}{\{x_\lambda\}_{\lambda\in D}}
\newcommand{\flam}{\{f_\lambda\}_{\lambda\in D}}
\newcommand{\flamx}{\{f_\lambda(x)\}_{\lambda\in D}}
\newcommand{\dt}{{d_{_\Pi}}}

\begin{frontmatter}	

\title{ Cartesian closed and stable subconstructs of $[0,1]$-$\mathbf{Cat}$}

 \author{Hongliang Lai \corref{cor}}
\ead{hllai@scu.edu.cn}

\author{Qingzhu Luo}
\ead{luoqingzhu@foxmail.com }

\cortext[cor]{Corresponding author.}
\address{School of Mathematics, Sichuan University, Chengdu 610064, China}
  
\begin{abstract}
Let $\with$ be a continuous triangular norm on the unit interval $[0,1]$ and $\mathbf{A}$ be a cartesian closed and stable subconstruct of the category consisting of all real-enriched categories. Firstly, it is shown that the category $\mathbf{A}$ is cartesian closed if and only if it is determined by a suitable subset $S\subseteq{M^2}$ of $[0,1]^2$, where $M$ is the set of all elements $x$ in $[0,1]$ such that $x\with x$ is idempotent. Secondly, it is shown that all Yoneda complete real-enriched categories valued in the set $M$ and Yoneda continuous $[0,1]$-functors   form a cartesian closed category.
\end{abstract}	

\end{frontmatter}

\section{Introduction}

Let $\&$ be a continuous triangular norm on the unit interval $[0,1]$. Then $\sQ=([0,1],\with,1)$ becomes a commutative and unital quantale \cite{Rosen1990}.  Categories enriched over $\sQ$ is called real-enriched categories \cite{Zhang2024}. As observed  by Lawvere \cite{Lawvere1973} in 1973, there is a deep connection between enriched categories and generalized logic. Particularly, for real-enriched categories, this generalized logic is the $BL$-logic  developed by H\'{a}jek \cite{Hajek1998}.  Therefore,
real-enriched categories may be viewed as ordered sets in the sense of $BL$-logic. Real-enriched categories encompass  preordered sets,  generalized metric spaces and fuzzy preordered sets \cite{Klement2000,Zade1971} in a unified framework. Naturally, they  become  the fundamental objects for quantitative domain theory, see e.g.  \cite{America1989,Antoniuk2011,BBR1998,Flagg1997,Gutierres2013,Hofmann2011,Hofmann2018,Kunzi2002,Lai2016}.   

Cartesian closed categories  are of particular interest in domain theory \cite{Gierz2003}, as explained: ``One of the most noteworthy features of $\mathbf{DCPO}$ is that it is cartesian closed. Not only is this fundamental for the application of continuous lattices and domains to logic and computing, but it also provides evidence of the mathematical naturalness of the notion.''  When the quantitative domain theory is concerned,  Yoneda complete quantale-enriched categories, in which all forward Cauchy nets (or forward Cauchy sequences) are convergent,  are often treated as quantitative preordered sets satisfying the directed completeness \cite{BBR1998,Fan2001,  Flagg1997,Flagg2002,Lai2006,Wagner1994,Wagner1997}.    However, unlike its classical counterpart,  the category $\mathbf{YCCat}$ consisting of Yoneda complete real-enriched categories and Yoneda  continuous  $[0,1]$-functors  need not be cartesian closed in general. In fact, it is cartesian closed if and only   if the continuous triangular norm $\with=\wedge$ \cite [Theorem 4.7]{Lai2016}. 

 Consider the category $\RCat$ consisting of all real-enriched categories and $[0,1]$-functors. Since   $\mathbf{YCCat}$ is a reflective subcategory of $\RCat$, Yoneda complete real-enriched categories $A$ and $B$  share the product object in   both   $\RCat$ and $\mathbf{YCCat}$.  If the power object $B^A$ in $\RCat$  exists, then their power object in $\mathbf{YCCat}$  is a subobject of $B^A$. However, it is not easy to seek the power object of $A$ and $B$ in  $\RCat$. It is even unfortunate that the power object in $\RCat$ may fail to exist in general \cite{Clem03,Clem06}. Therefore, we   give up the whole category $\RCat$ but consider its cartesian closed and full subcategory $\mathbf{A}$. 
Since the  power objects $B^A$ do exist for all $A,B$ in $\mathbf{A}$, one can seek the   power object $[A\ra B]$ of Yoneda complete real-enriched categories $A$ and $B$ in $\mathbf{A}\cap\YCom$, as a subobject of the power object $B^A$  in $\mathbf{A}$.


In this paper, firstly, we find out all cartesian closed and stable subcategories of $\RCat$. Observe that the category $\RCat$ contains a trivial cartesian closed and stable subcategory $\Ord$, consisting of all crisp preordered sets and order preserving maps.  Moreover,   in a particular case, let $\sLu=([0,1],\with_{\Lu},1)$, where $\with_\Lu$ is the {\L}ukasiewicz triangular norm, and $\sLu_3=(\{0,0.5,1\},\with_\Lu,1)$ be the subquantale of $\sLu$, then the category $\sLu\text{-}\Cat$ has a stable  subcategory $\sLu_3$-$\Cat$. The category $\sLu_3$-$\Cat$ is cartesian closed \cite [Example 4.8]{Lai2016}. Therefore, one obtains a chain consisting of cartesian closed categories in the below: 
\[
\Ord\subseteq\sLu_3\text{-}\Cat\subseteq\sLu\text{-}\Cat.
\]
Thus, one can see that there may be  non-trivial cartesian closed and stable subcategories of $\RCat$ which is larger than  $\Ord$ even if the quantale $\sLu$ is not a frame.  In fact,  we characterize explicitly all  cartesian closed and stable subcategories of $\RCat$, which are determined by suitable subsets $S\subseteq M^2$ of $[0,1]^2$, where $M=\{a\in[0,1]\mid a\& a \text{ is idempotent}\}$.  

Secondly, we show that the category $\YComM$, which consists of all Yoneda complete real-enriched categories valued in the set $M$ and Yoneda continuous $[0,1]$-functors, is cartesian closed. Thus, one obtains a chain of cartesian closed categories:
$$\mathbf{DCPO}\subseteq\YComK\subseteq\YComM,$$
where $(K,\with,1)$ is a complete subquantale of $(M,\with,1)$.
  
The contents are arranged as follows.
 Section 2: Some basic concepts about   real-enriched categories  are recalled.
Section 3:  Describe all  stable subcategories of $\RCat$. 
Section 4: Characterize  all cartesian closed and stable subcategories  of $\RCat$. Section 5: Prove that for each complete subquantale $(K,\with,1)$ of $(M,\with,1)$,    the category $\YComK$ is  cartesian closed.

\section{Real-enriched categories}

In this section, some basic notions of real-enriched categories are recalled. Terminologies and notations are mainly from \cite{Zhang2024}.

A   \emph{triangular norm} \cite{Klement2000} (\emph{ t-norm} for short)    is a  binary operation $\with$ on  the unit interval $[0,1]$ such that for all $x,y,z\in [0,1]$,  
\begin{itemize}
    \item[(1)] $x\with y=y\with x;$
    \item[(2)] $x\with(y\with z)=(x\with y)\with z;$
    \item [(3)]$x\with y\leq x\with z \text{ whenever } y\leq z;$
    \item [(4)]$x\with 1=x$.
\end{itemize}
That is, the triple $([0,1],\with, 1)$ is a commutative po-monoid \cite{Birk1973}. A t-norm $\with$ is called to be \emph{continuous} if $\with: [0, 1]\times [0, 1]\lra [0, 1]$ is a continuous function with respect to the usual topology. In this case, the function $x\with-:[0,1]\lra [0,1]$ preserves arbitrary joins for each $x\in [0,1]$, hence the triple $([0,1],\with,1)$ is also a quantale \cite{Rosen1990}.


There are three basic continuous t-norms listed in the below:
\begin{itemize}
\item the G\"{o}del t-norm: $x\with_G y=x\wedge y$;
\item the product t-norm $x\with_P y=x\cdot y$; 
\item the  {\L}ukasiewicz t-norm: $x\with_{\text{\L}} y=\max\{x+y-1, 0\}$.
\end{itemize}

One can obtain all continuous t-norms by ordinal sums of the product t-norm and {\L}ukasiewicz t-norm. 

\begin{thm}(\cite{Klement2000})\label{Ordinal sum}
 Every continuous t-norm
$\&$ on $[0, 1]$ is an ordinal sum of the product t-norm and the {\L}ukasiewicz t-norm. That means,   there is a countable set of disjoint open intervals 
$\{(a_i, b_i)\}_{i\in I}$ in $[0, 1]$ such that
\begin{enumerate}
\item[{\rm (i)}] for each $i\in I$, both $a_i$ and $b_i$ are idempotent and there is an order isomorphism $\varphi_i:[0,1]\lra[a_i,b_i]$ satisfying either
 $\varphi_i(x\cdot y)=\varphi_i(x)\with\varphi_i(y)$ or $\varphi_i(x\with_{\text{\L}}y)=\varphi_i(x)\with\varphi_i(y)$;
\item[{\rm (ii)}] $x \& y =x\wedge y$ whenever $(x, y)\not\in\bigcup_{i\in I}(a_i,b_i)^2$.
\end{enumerate}
 
Each $[a_i, b_i]$ with the restriction of $\&$ is called an Archimedean block of $\&$.
\end{thm} 

\begin{lem}(\cite{Klement2000})\label{contain idem}
    Let $\with$ be a continuous t-norm on $[0,1]$ and $p$ be an idempotent element of $\with$. Then $x\with y=x\wedge y$ whenever $x\leq p\leq
    y.$
\end{lem}

\begin{SA} In the rest of this paper, $\with$ is always assumed to be a continuous t-norm.\end{SA}

\begin{defn}(\cite{Zhang2024})
    Let $\with$ be a continuous t-norm. A \emph{real-enriched category}, or a \emph{$[0,1]$-category} for short, is a pair $(X,r)$, where $X$ is a set and $r: X^2 \lra  [0,1]$ is a map such that 
    \begin{itemize}
        \item [(i)] $r(x,x)=1$ for all $x\in X$;
        \item [(ii)] $r(y,z) \with r(x,y) \leq r(x,z)$ for all $x,y,z \in X$. 
    \end{itemize}
\end{defn}
A $[0,1]$-category $(X,r)$ is \emph{symmetric} if $r(x,y)=r(y,x)$ for all $x,y\in X.$
If one interpret the value $r(x, y)$ as the truth degree that $x$ is smaller than or
equal to $y$, then a real-enriched category
can be regarded as a quantitative preordered set in Lawvere's sense \cite{Lawvere1973}, or a fuzzy preordered set in the fuzzy set theory \cite{Belo2002,Klement2000}.


 Suppose both $(X,r)$ and $(Y,s)$ are $[0,1]$-categories. A map $f:X\lra Y$  is called a \emph{$[0,1]$-functor} if  
$$r(x_1,x_2) \leq s(f(x_1),f(x_2))$$ for all $x_1$, $x_2$  in $X$. All $[0,1]$-categories  and  $[0,1]$-functors form a category, denoted by  
 $$\RCat.$$  

In the following example, it is shown that preordered sets and generalized metric spaces can all be regarded as  real-enriched categories. 

\begin{exmp}
\begin{itemize}
\item[(1)] Let $(X,\leq)$ be a preordered set. Define a map $r_\leq:X\times X\lra[0,1]$ by  
\[r_\leq(x,y)=\begin{cases}
		0,&x\nleq y,\\
		1,&x\leq y.
		\end{cases} \]
 Clearly, $(X,r_\leq)$ is a real-enriched category. Moreover, the assignment  $(X,\leq)\mapsto(X,r_\leq)$ on objects gives a full embedding functor
 \[\pi:\Ord\lra\RCat,\]
 where $\Ord$ is the category of all preordered sets and order preserving maps.

\item[(2)] A map $d:X\times X\lra[0,\infty]$ is called a generalized metric \cite{Lawvere1973} on the set $X$ if it satisfies that  
\begin{enumerate}
\item[(i)] $d(x,x)=0$ for all $x\in X$,
\item[(ii)] $d(x,y)+d(y,z)\geq d(x,z)$ for all $x$, $y$ and $z$ in $X$.
\end{enumerate}
A map $f$ between generalized metric spaces $(X,d_1)$ and $(Y,d_2)$ is called to be non-expansive if 
\[d_1(x,y)\geq d_2(f(x),f(y)).\]
All generalized metric spaces and non-expansive maps form a category, denoted by $\mathbf{GMet}$.
Let $\&$ be the product t-norm, since the map $x\mapsto -\ln x$ is an isomorphism between the po-monoids  $([0,1],\cdot,1)$ and $([0,\infty]^{\mathrm{op}},+,0)$, it induces an isomorphism between the category $\RCat$ and $\mathbf{GMet}$.  In this case,   a real-enriched category $(X,r)$ is assigned to a generalized metric space $(X, -\ln r)$. 
\end{itemize}
\end{exmp}

\section{Stable subconstructs of $\RCat$}

In this section, we consider the stable subconstructs of the concrete category $\RCat$. Our main technique is similar to that in \cite{Lowen2015} for quasi-metric spaces. We  make some slight changes for real-enriched categories.
   
Recall from \cite{Acc} that a construct is a concrete category $(\bfA,U)$ over $\Set$ with $U$ being the forgetful functor.  A subcategory $\bfA$ of a construct $\bfB$ is called a subconstruct. A subconstruct $\bfA$ of $\bfB$ is called \emph{concretely reflective} in $\bfB$ if for each $\bfB$-object there is an identity-carried $\bfA$-reflection arrow. Reflectors induced by identity-carried reflection arrows are called concrete reflections. 
A concretely reflective subcategory of an amnestic concrete category is always full \cite[Proposition 5.24]{Acc}.
Concretely  coreflective subconstructs and concrete coreflections are defined dually. 

A subconstruct $\bfA$ of a construct $\bfB$ is  \emph{stable} \cite{Lowen2015} if $\bfA$ is simultaneously concretely reflective and concretely coreflective in $\bfB$.  

Let $(\bfA,U)$ be a construct, a \emph{source} is a family of morphisms $\{f_i: A\lra{A_i}\}_{i\in I}$ with domain $A$ and codomain $A_i$. A source $\{A\lra{A_i}\}_{i\in I}$ in $\bfA$ is \emph{initial} provided that a set-map $f: UB\lra{UA}$ is an $\bfA$-morphism whenever each composite $f_i\circ f: UB\lra{UA_i}$ is an $\bfA$-morphism. Dually, a \emph{sink} is a family of morphisms $\{f_i: A_i\lra{A}\}_{i\in I}$ with domain $A_i$ and codomain $A$. A sink $\{A_i\lra{A}\}_{i\in I}$ in $\bfA$ is \emph{final} provided that a set-map $f: UA\lra{UB}$ is an $\bfA$-morphism whenever each composite $f\circ f_i: UA_i\lra{UB}$ is an $\bfA$-morphism. 

A construct $(\bfA,U)$ is \emph{topological} \cite{Acc} if every $U$-structured source $\{f_i: X\lra{UA_i}\}_{i\in I}$ has a unique initial lift $\{f_i: A\lra{A_i}\}_{i\in I}$ in $\bfA$, which implies the existence of unique final lifts of $U$-structured sinks.
\begin{exmp}
    The category $\RCat$ is topological.  For a structured source $\{f_i: X\lra{(X_i,r_i)}\}_{i\in I}$  in $\Set$, the  initial lift in $\RCat$ is given by $\{f_i: (X,d_{in})\lra{(X_i,r_i)}\}_{i\in I}$ with $$d_{in}(x,y)=\bw_{i\in I} r_i(f_i(x),f_i(y)).$$
    For a structured sink $\{f_i: (X_i,r_i)\lra X\}_{i\in I}$  in $\Set$, the  final lift in $\RCat$ is given by $\{f_i: (X_i,r_i)\lra (X,d_{fin})\}_{i\in I}$ with $$d_{fin}(x,y)=\bw\{s(x,y)\mid (X,s)\text{ is a $[0,1]$-category with } r_i(x,y)\leq s(f_i(x),f_i(y)) \text{ for all } i\in I \}.$$  
\end{exmp}

A full subconstruct $\bfA$ of a construct $\bfB$ is \emph{initially closed}  if every initial source  in $\bfB$ whose codomain is a family of $\bfA$-objects has its domain  in $\bfA$. And a full subconstruct $\bfA$ of a construct $\bfB$ is \emph{finally closed}  if every final sink  in $\bfB$ whose domain is a family of $\bfA$-objects has its codomain  in $\bfA$.

\begin{prop}(\cite [Proposition 21.31]{Acc})\label{icfc}
  For any full subconstruct $\bfA$ of a topological construct $\bfB$ the following conditions are equivalent:
  \begin{itemize}
      \item [(1)] $\bfA$ is a stable subconstruct of $\bfB$.
      \item [(2)] $\bfA$ is both initially closed and finally closed in $\bfB$.
  \end{itemize}
\end{prop}


The category $\RCat$ contains $\Ord$   as a stable subconstruct. 
In fact, the  embedding functor 
 \[\pi:\Ord\lra\RCat\]
 is simultaneously  concretely reflective and concretely coreflective. 
\begin{center}
\begin{tikzpicture}
\draw[>->](0,0)node[left]{$\Ord$}--(1.5,0)node[right]{$\RCat$};
\draw[->>](1.5,.4)--(0,.4);
\node at(.75,.4)[above]{$\sigma$};
\node at(.75,-.05)[above]{$\bot$};
\draw[->>](1.5,-.4)--(0,-.4);
\node at(.75,-.45)[above]{$\bot$};
\node at(.75,-.4)[below]{$\rho$};
\end{tikzpicture}
\end{center}
For each $[0,1]$-category $(X,r)$, the coreflection $\rho(X)=(X,\leq_{\rho(r)})$ is equipped with the greatest preorder below $r$ given by 
$$x\leq_{\rho(r)}y \iff r(x,y)=1,$$
while the reflection $\sigma(X)=(X,\leq_{\sigma(r)})$ is equipped with the smallest preorder on $X$ above $r$, which is the transitive closure of the reflexive relation 
 $\leq_{\sigma'(r)}$ given by $$x\leq_{\sigma'(r)}y\iff r(x,y)\neq 0.$$


 Now we describe all stable subconstructs of $\RCat$. Let $(\bfA,U)$ be a construct, the \emph{fibre} $\bfA(X)$ of a set $X$ is the class consisting of all $\bfA$-objects $A$ with  $UA=X$ ordered by $A\leq B$ if and only if $\id_X:A\lra B$ is an $\bfA$-morphism. The construct $\RCat$ is \emph{fibre-small}, that is,  the fibre $\RCat(X)$ of each set $X$ is small. Our main idea is based on the fact that  the fibre $\RCat(\bbt)$ of the set $\bbt=\{0,1\}$ is  finally dense in $\RCat$.  So, the stable subconstructs of $\RCat(\bbt)$ determine all stable subconstructs of $\RCat$. 

Let $(\bbt,r)$ be a $[0,1]$-category. Then $r$ is completely determined by the values of $r(0,1)=a$ and  $r(1,0)=b$ since $r(0,0)=r(1,1)=1$ holds always. Conversely, every pair $(a,b)\in [0,1]^2$ determines a two-point $[0,1]$-category $(\bbt,r)$ with $r(0,0)=r(1,1)=1,r(0,1)=a$ and $r(1,0)=b$. Thus, by abuse  of notation, we often write the $[0,1]$-category $(\bbt, r)$ as $(\bbt,a,b)$ for a pair $(a,b)\in [0,1]^2$. That is, there is an order isomorphism $$\RCat(\bbt)\cong [0,1]^2.$$  Moreover, the fibre $\bfA(\bbt)$ can be viewed as a subset of $[0,1]^2$ for any subconstruct $\bfA$ of $\RCat$. 



\begin{prop}\label{suit-set} Let $\with$ be a continuous t-norm. If $\bfA$ is a stable subconstruct of $\RCat$, then the fibre $\bfA(\bbt)$ satisfies the following properties: 
\begin{itemize}
\item[{\rm (S1)}]    $(a_i,b_i)\in \bfA(\bbt)$ $\Lra$   $\bigvee(a_i,b_i)\in\bfA(\bbt)$, $\bigwedge(a_i,b_i)\in\bfA(\bbt)$;
\item[{\rm (S2)}]  $(a,b)\in\bfA(\bbt)$ $\Lra$ $(b,a)\in\bfA(\bbt)$;
\item[{\rm (S3)}]  $(a_1,b_1)\in\bfA(\bbt)$, $(a_2,b_2)\in\bfA(\bbt)$ $\Lra$  $(a_1\with a_2,b_1\with b_2)\in\bfA(\bbt)$.
\end{itemize}
\end{prop}

\begin{proof}
(S1): It can be easily verified that the sink $\{\id_\bbt: (\bbt, a_i,b_i)\lra{(\bbt,a,b)}\}_{i\in I}$ with $a=\bigvee a_i$ and $b=\bigvee b_i$ is final in $\RCat$, since $\bfA$ is finally closed by Proposition \ref{icfc}, it follows that $(\bv a_i,\bv b_i)\in \bfA(\bbt)$ whenever all $(a_i,b_i)\in\bfA(\bbt)$.  Dually, since the source $\{\id_\bbt: (\bbt, a,b)\lra{(\bbt,a_i,b_i)}\}_{i\in I}$ with $a=\bigwedge a_i$ and $b=\bigwedge b_i$ is initial in $\RCat$ and $\bfA$ is initially closed, it follows that $(\bigwedge a_i,\bigwedge b_i)\in \bfA(\bbt)$ whenever all $(a_i,b_i)\in\bfA(\bbt)$.

(S2): Consider the map $\tau:\bbt\lra\bbt$, $\tau(0)=1$ and $\tau(1)=0$. Then $\tau:(\bbt,a,b)\lra{(\bbt,b,a)}$ is a final sink in $\RCat$. Thus, one obtains that $(b,a)\in\bfA(\bbt)$ if  $(a,b)\in\bfA(\bbt)$.

(S3): Let $X=\{x,y,z\}$. Consider maps $f_1:\bbt\lra X, f_1(0)=x, f_1(1)=y$ and $f_2:\bbt\lra X, f_2(0)=y,f_2(1)=z$. Since $\RCat$ is topological, the sink $\{f_i: (\bbt,a_i,b_i)\lra X|i=1,2\}$ has a unique final lift $\{f_i: (\bbt,a_i,b_i)\lra(X,d_{fin})|i=1,2\}$  in $\RCat$. Clearly, it holds that $d_{fin}(x,z)=a_1\with a_2$ and $d_{fin}(z,x)=b_1\with b_2$. Let $h:(\bbt,a_1\with a_2,b_1\with b_2)\lra (X,d_{fin})$ be a source with $h(0)=x$ and $h(1)=z$, it is initial in $\RCat$. Thus, if both $(a_1,b_1)$ and $(a_2,b_2)$ are in $\bfA(\bbt)$, then both $(X,d_{fin})$ and $(\bbt,a_1\with a_2,b_1\with b_2)$ are also in $\bfA$. Hence, $(a_1\with a_2,b_1\with b_2)\in\bfA(\bbt)$ as desired.
\end{proof}


 
 A subset $S$ of $[0,1]^2$ satisfying the above conditions $(S1)-(S3)$ is called  a (symmetric) \emph{suitable }\cite{Lowen2015} subset of $[0,1]^2$.  Given a suitable subset $S$ of $[0,1]^2$,  collect all objects $(X,r)$ in $\RCat$ such that 
$$\forall x,y\in X,\ (r(x,y),r(y,x))\in S,$$ then one obtains a full subconstruct  of  $\RCat$, denoted by $$\RCat_{S}.$$

\begin{exmp} Let $\with$ be a continuous t-norm. We list some suitable subsets of $[0,1]^2$ and the  subconstructs $\RCat_S$ of $\RCat$.

\begin{enumerate}[label=\arabic*.]
\item[(1)]  	A subset $K\subseteq[0,1]$ is called a \emph{complete sublattice} of $[0,1]$ if $K$ is closed under arbitrary joins and meets.  If $K$ is at the same time closed under the binary operation $\with$, then the triple
$(K,\with,1)$ is a quantale, called a complete subquantale of $([0,1],\with,1)$. A $[0,1]$-category $(X,r)$ satisfying that $r(x,y)\in K$ for all $x,y\in X$ is called a $K$-category. Thus, all $K$-categories constitute a full subcategory of $\RCat$, which is denoted by $\KCat$.

Let $(K,\with,1)$ be a complete subquantale of $([0,1],\with,1)$, then the product $K^2$   is a suitable subset of $[0,1]^2$. Clearly, one has that a real-enriched category $(X,r)\in\RCat_{K^2}$ if and only if $r(x,y)\in K$ for all $x$ and $y$ in $X$. Therefore, one has that 
 $$\RCat_{K^2}=\KCat.$$ 
Particularly, when $K=\bbt=\{0,1\}$,   one has that 
  \[\RCat_{\bbt^2}=\bbt\text{-}\Cat\cong\Ord.\]

Let $K_\Delta=\{(p,p)|p\in K\}$. Then $K_\Delta$  is also a suitable subset of $[0,1]^2$. In this case, a 
 $[0,1]$-category $(X,r)\in\RCat_{K_\Delta}$ if and only if $(X,r)\in\KCat$ and $r(x,y)=r(y,x)$ for all $x,y\in X$, that is, $(X,r)$ is a symmetric $K$-category. 

\item[(2)] Since $\with$ is a continuous t-norm, for each $x\in [0,1]$, the set $\{z\in[0,1]\mid z\&z\leq x\}$ has the maximum.  Define the square root of $x$ by  $$x_{_\&}^{^\frac{1}{2}}=\max\{z\in[0,1]|z\with z\leq x\},$$
then 
$$S=\{(x,y)\in[0,1]^2|x\with x\leq y\leq x_{_\&}^{^\frac{1}{2}}\}$$
is a suitable subset of $[0,1]^2$.  It determines a stable  subconstruct $\RCat_S$ whose objects are those $[0,1]$-categories $(X,r)$  satisfying 
 $$r(x,y)\with r(x,y)\leq r(y,x)\leq(r(x,y))_{_\&}^{^\frac{1}{2}}$$
 for all $x,y\in X$. 
\end{enumerate}

\end{exmp}

\begin{prop}\label{stable sub}
Let $\with$ be a continuous t-norm.  A  full subconstruct $\bfA$ of $\RCat$ is stable if and only if there is a suitable subset $S$ of $[0,1]^2$ such that $\bfA=\RCat_S$.
\end{prop}
\begin{proof}
The proof is similar to that in \cite[Theorem 12.1.14]{Lowen2015} for quasi-metric spaces. We present here a sketch of the proof  for the convenience to readers. 

\emph{Necessity}: By Proposition \ref{suit-set}, $\bfA(\bbt)$ is a suitable subset of $[0,1]^2$. We check that $\bfA=\RCat_{\bfA(\bbt)}$. Let  $(X,r)$ be a $[0,1]$-category. For all $x$ and $y$ in $X$, define a map $$f_{xy}:\bbt\lra X,\quad f_{xy}(0)=x \text{ and } f_{xy}(1)=y.$$ 
 On one hand, one can see that each arrow 
 $$f_{xy}:(\bbt,r(x,y),r(y,x))\lra(X,r)$$ is an initial source in $\RCat$. Thus, the pair $(r(x,y),r(y,x))$ is in $\bfA(\bbt)$ whenever $(X,r)$ is in $\bfA$, which implies that the $[0,1]$-category $(X,r)$ is also in $\RCat_{\bfA(\bbt)}$. On the other hand, one can see that
 $$\{f_{xy}:(\bbt,r(x,y),r(y,x))\lra(X,r)|x,y\in X\}$$ is a final sink in $\RCat$. Thus $(X,r)$ is in $\bfA$ whenever $(X,r)$ lies in $\RCat_{\bfA(\bbt)}$.  
 

\emph{Sufficiency}: Let $S$ be a suitable subset of $[0,1]^2$, we  show that  $\RCat_S$ is stable in $\RCat$. 

Firstly, for a $[0,1]$-category $(X,r)$ in $\RCat$, let 
$$C(r)(x,y)=\bigvee\{p\in Q|(p,q)\in S\text{ and } (p,q)\leq (r(x,y),r(y,x))\}$$ for all $x$ and $y$ in $X$.  It is clear that $(C(r)(x,y),C(r)(y,x))\in S$ for all $x,y\in X$. Notice that 
$$C(r)(x,y)\with C(r)(y,z)\leq r(x,y)\with r(y,z)\leq r(x,z)$$
 and the pair $$(C(r)(x,y)\with C(r)(y,z),C(r)(z,y)\with C(r)(y,x))$$ lies in $S$ for all $x$, $y$ and $z$ in $X$. Thus, one has that 
 $$C(r)(x,y)\with C(r)(y,z)\leq C(r)(x,z),$$  that is, $(X,C(r))$ lies in $\RCat_S$. Therefore,  
 $$\id_X:(X,C(r))\lra(X,r)$$ is the coreflection arrow.
 
 Secondly, for a $[0,1]$-category $(X,r)$, let 
 $$R(r)(x,y)=\bigwedge\{s(x,y)| (X,s)\in\RCat_S\text{ and } r(x,y)\leq s(x,y)\}$$ for all $x$ and $y$ in $X$. One can easily check that $(X,R(r))$ is also in $\RCat_S$. Therefore,
 $$\id_X:(X,r)\lra(X,R(r))$$ is the reflection arrow. 
\end{proof}

\begin{cor}
  A stable subconstruct $\bfA$ of $\RCat$ contains $\Ord$ as a subconstruct if and only if there is a complete subquantale $(K,\with,1)$ of $([0,1],\with,1)$  such that $\bfA=\KCat$.
\end{cor}
\begin{proof}
It suffices to show the necessity. Let $\bfA$ be a stable subconstruct of $\RCat$ containing $\Ord$, by the above Proposition \ref{stable sub},  $\bfA=\RCat_S$ for some suitable subset $S$ of $[0,1]^2$. Then it holds that $\bbt^2=\Ord(\bbt)\subseteq \bfA(\bbt)=S$. Define $$K=\{a\mid(a,b)\in S\},$$ it is clearly a complete sublattice of 
$[0,1]$ and $(K,\with,1)$ is a complete subquantale of $([0,1],\with,1)$. Moreover, $S\subseteq{K^2}$.  For each $(a,b)\in K^2$, there exists $m$ and $n$ with $(a,m)\in S$ and $(b,n)\in S$ such that \[(a,b)=((a,m)\wedge(1,0))\vee((n,b)\wedge(0,1)) \in S\] since $S$ is suitable, it follows that $K^2=S$. Hence $\bfA=\RCat_{K^2}=\KCat$.
\end{proof}


\section{Closed stable subconstructs of $\RCat$}

Let $(X,r)$ and $(Y,s)$ be $[0,1]$-categories. 
Define 
\[r\otimes s((x,y),(x',y'))=r(x,x')\with s(y,y')\] for all $(x,y), (x',y')\in X\times Y$. Then $(X\times Y,r\otimes s)$ is a $[0,1]$-category, called the  \emph{tensor product} of $(X,r)$ and $(Y,s)$, and is denoted by $(X,r)\otimes (Y,s)$.  

Denote the set of all $[0,1]$-functors from $(X,r)$ to $(Y,s)$ by $[(X,r),(Y,s)]$. Equipped with the real-enriched relation 
$$d_\otimes:[(X,r),(Y,s)]^2\to[0,1], \quad 
d_\otimes(f,g)=\bw_{x\in X}s(f(x),g(x)),$$
then $([(X,r),(Y,s)],d_\otimes)$ is a $[0,1]$-category.
For a given $[0,1]$-category $(X,r)$,  the assignment $(Y,s)\mapsto ([(X,r),(Y,s)],d_\otimes)$ is a functor  on $\RCat$, which is the right adjoint of $(X,r)\otimes(-)$.   Therefore, together with the terminal object $\vmathbb{1}=(\{\star\},e)$ where $e(\star,\star)=1$, $(\RCat,\otimes,\vmathbb{1})$ is a monoidal closed category \cite{Wagner1997}. 

 \begin{prop} \label{stab are monclo} All  stable subconstructs of $\RCat$ are    monoidal closed.
\end{prop}
\begin{proof} Let $\mathbf{A}$ be a subconstruct of $\RCat$, then there is a suitable subset $S$ of $[0,1]^2$ such that $\mathbf{A}=\RCat_S$. For $[0,1]$-categories $(X,r)$ and $(Y,s)$ in $\mathbf{A}$, that is, $(r(x,x'),r(x',x))\in S$ and $(s(y,y'),s(y',y))\in S$ for all $x,x'\in X$ and $y,y'\in Y$. Then the pair 
$$(r\otimes s((x,y),(x',y')),r\otimes s((x',y'),(x,y))=(r(x,x'),r(x',x))\with(s(y,y'),s(y',y))\in S$$
for all $(x,y), (x',y')\in X\times Y$. And the pair \begin{align*}
    (d_\otimes(f,g),d_\otimes(g,f))&=\Big(\bw_{x\in X}s(f(x),g(x)),\bw_{x\in X}s(g(x),f(x))\Big)\\
    &=\bw_{x\in X}(s(f(x),g(x)),s(g(x),f(x)))\in S
\end{align*} for all $f,g\in[(X,r),(Y,s)]$.
Hence the tensor product $(X,r)\otimes (Y,s)$ and the function space $([(X,r),(Y,s)],d_{\otimes})$ are both in $\mathbf{A}$. Since the terminal object $\vmathbb{1}$ is contained in $\bfA$ trivially, $(\bfA,\otimes,\vmathbb{1})$ is monoidal closed.
\end{proof}

In the category $\RCat$, one can also consider the cartesian product of $[0,1]$-categories $(X,r)$ and $(Y,s)$. Define 
 $$r\times s((x,y),(x',y'))=r(x,x')\wedge s(y,y')$$ for all $(x,y),(x',y')\in X\times Y$, then $(X\times Y,r\times s)$ is the product of $(X,r)$ and $(Y,s)$ in $\RCat$. However, the category $\RCat$ is cartesian closed if and only if the t-norm $\with=\wedge$ \cite[Theorem 4.7]{Lai2016}, hence, it need not be cartesian closed in general. Our aim is to find out  the cartesian closed ones among all stable subconstructs of $\RCat$.

Recall from \cite{Acc} that a category $\bfA$ is \emph{cartesian closed }if it has finite products and for
each $\bfA$-object $A$ the functor $A\times -: \bfA\lra\bfA$ has a right adjoint. The right adjoint functor
for $A\times -$ is denoted on objects by $B\mapsto B^A$,  called power objects, and the associated co-universal arrows
are denoted by
$$\ev: A\times B^A\lra B, $$
  called evaluation morphisms.


\begin{lem}(\cite [Proposition 27.9]{Acc})\label{stable of ccc}
	Any isomorphism-closed reflective and coreflective full subcategory of a cartesian closed category is again cartesian closed.
\end{lem}

To describe the structure map of the power objects in $\RCat$, we need to introduce the \emph{residuated} operation $\ra$ on $[0,1]$ that corresponds to $\wedge$. For all $x$ and $y\in [0,1]$, define  $$x\ra y=  \begin{cases}
		y,&x> y,\\
		1,&x\leq y.
		\end{cases} $$ 
  Then the binary operations $\wedge$ and $\ra$ satisfy the adjoint condition
$$x\wedge y\leq z\iff y\leq x\ra z$$ 
 for all $x$, $y$ and $z$ in $[0,1]$. By this adjoint condition, one can easily see the  properties in the below.   
\begin{prop}
\label{resd prop}
For all $x,x_i\in[0,1]$, it holds that
 \[x \ra\Big(\bw_{i \in I}x_i\Big) = \bw_{i \in I}(x \ra x_i),\quad 
\Big(\bv_{i \in I}{x_i}\Big) \ra x = \bw_{i \in I}(x_i \ra x).\]
\end{prop}

\begin{prop}\label{powobj} 
 Suppose   $(X,r)$ and $(Y,s)$ are $[0,1]$-categories such that their power object $(Y,s)^{(X,r)}$ exists in the category $\RCat$. Then the following statements hold.
 \begin{itemize}
 \item[(1)] (\cite{Clem03,Clem06}) The power object $(Y,s)^{(X,r)}$ has an underlying set 
 $$|(Y,s)^{(X,r)}|=[(X,r),(Y,s)]$$ 
 and is equipped with the structure map 
 $$\dt:[(X,r),(Y,s)]^2\lra{K},\quad \dt(f,g)=\bigwedge_{x,y\in X}r(x,y)\ra s(f(x),g(y)).$$
 \item[(2)] For all $[0,1]$-functors $f$ and $g$ in $[(X,r),(Y,s)]$, $\dt(f,g)\leq d_\otimes(f,g)$.
 \item[(3)] If $(K,\with,1)$ is a complete subquantale of $([0,1],\with,1)$, $(X,r)$ and $(Y,s)$ are both $K$-categories, then the power object $(Y,s)^{(X,r)}$ is also a $K$-category. 
 \end{itemize}
\end{prop} 
\begin{proof}
    $(2):$ Let $f,g\in [(X,r),(Y,s)]$. It holds that 
    \begin{align*}
        \dt(f,g)&=\bw_{x,y\in X}r(x,y)\ra s(f(x),g(y))\\
        &\leq\bw_{x\in X}r(x,x)\ra s(f(x),g(x))\\
        &=\bw_{x\in X}s(f(x),g(x))\\
        &=d_\otimes(f,g).
    \end{align*}
    $(3):$ It is obvious from the definition of $\ra$.
\end{proof}

For a $K$-category $(X,r)$, the functor $X\times-:\KCat\lra\KCat$ has a right adjoint if and only if the  $[0,1]$-functor from $(X,r)$ to the terminal object $\vmathbb{1}$  has a right adjoint in the sense of \cite{Clem03,Clem06,Tholen1987}.  Therefore, as a special case of the Theorem 3.4 in \cite{Clem06}, one obtains the following proposition. 
 
\begin{prop} (\cite{Clem06})\label{exp ob}
	 Suppose $(K,\with,1)$ is a complete subquantale of $([0,1],\with,1)$ and $(X,r)$ is a $K$-category. The functor $(X,r)\times(-):\KCat\lra\KCat$ has a right adjoint if and only if it holds that
 $$(u\with v)\wedge r(x,y)\leq\bigvee_{z\in X}(u\wedge r(x,z))\with(v\wedge r(z,y))$$ 
 for all $u,v\in K$ and $x,y\in X$.
\end{prop}
\begin{thm}\label{ccc}
    Let $(K,\with,1)$ be a complete subquantale of $([0,1],\with,1)$. The following statements are equivalent:
 \begin{enumerate} 
 \item[(1)] the category $\KCat$ is cartesian closed;
 \item[(2)] the category $\RCat_{K_\Delta}$ is cartesian closed;
 \item[(3)] for all $u,v,r\in K$, it holds that  $$(u\with v )\wedge r=((u\wedge r)\with v)\vee ((v\wedge r)\with u);$$ 
 \item[(4)] $a\with a$ is idempotent for all $a\in K$.
\end{enumerate}
\end{thm}
\begin{proof}
$(1)\Ra (2):$ It follows from that the category $\RCat_{K_\Delta}$ is reflective and coreflective in $\RCat$ since $K_\Delta$ is a suitable subset of $[0,1]^2$.

$(2)\Ra (3):$ Let $u,v,r\in K$. Consider the two-point $[0,1]$-category $(A,\rho)=(\bbt,r,r)$, since the category $\RCat_{K_\Delta}$ is cartesian closed, the functor\[(A,\rho)\times - :\RCat_{K_\Delta}\lra\RCat_{K_\Delta}\] has a right adjoint hence it preserves final sinks \cite{Acc}. Let $(B,s)=(\{x,z\},s)$ where $s(x,z)=s(z,x)=u$, $(C,t)=(\{z,y\},t)$ where $t(z,y)=t(y,z)=v$, and $(D,d)=(\{x,z,y\},d)$ where $d(x,z)=d(z,x)=u$, $d(z,y)=d(y,z)=v$ and $d(x,y)=d(y,x)=u\with v$. Then \begin{center}
    \begin{tikzcd}[scale=4]
	{(B,s)} \\
	&& {(D,d)} \\
	{(C,t)}
	\arrow["{i_B}", hook, from=1-1, to=2-3]
	\arrow["{i_C}"', hook', from=3-1, to=2-3]
\end{tikzcd}
\end{center}
 is a clearly a final sink in $\RCat_{K_\Delta}$ where $i_B$ and $i_c$ are the corresponding embedding. It follows that 
\begin{center}
    \begin{tikzcd}[scale=4]
	(A\times B,\rho\times s) \\
	&& (A\times D,\rho\times d) \\
	(A\times C,\rho\times t)
	\arrow["{\id_A\times i_B}", hook, from=1-1, to=2-3]
	\arrow["{\id_A\times i_C}"', hook', from=3-1, to=2-3]
\end{tikzcd}
\end{center}
is also a final sink in $\RCat_{K_\Delta}$. Since $$\rho\times{d}((0,x),(1,y))=r\wedge (u\with v),$$ and 
 \[ d_{fin}((0,x),(1,y))=((u\wedge r)\with v)\vee ((v\wedge r)\with u)\] where $(A\times{D},d_{fin})$ is the final lift of the structured sink $$\{A\times{B}\lra A\times{D},A\times{C}\lra{A\times{D}}\}.$$ Hence the equality $$(u\with v )\wedge r=((u\wedge r)\with v)\vee ((v\wedge r)\with u)$$ holds for all $u,v,r\in K.$

 $(3)\Ra (4):$ Let $u=v=a$ and $r=a\with a$, it follows that $$a\with a=a\with a\with a,$$  hence $a\with a$ is idempotent for all $a\in K.$

 $(4)\Ra (1):$ We need to show that the functor $(X,r)\times-:\KCat\lra\KCat$ has a right adjoint for every object $(X,r)$ in $\KCat$, equivalently, for all $u,v\in K$ and all $x$ and $y$ in $X$, it holds that $$(u\with v)\wedge r(x,y)\leq\bigvee_{z\in X}(u\wedge r(x,z))\with (v\wedge r(z,y)).$$
We check the inequality in two cases. 

\emph{Case 1}: There is an idempotent element $p$ between $u$ and $v$. Without loss of generality, let $u\leq p\leq v$.  In this case, $a\with v=a\wedge v=a$ for all $a\leq u$ by Lemma \ref{contain idem}, one has that 
\[
 \bigvee_{z\in X}(u\wedge r(x,z))\with (v\wedge r(z,y))\geq (u\wedge r(x,y))\with (v\wedge r(y,y)) = u\wedge r(x,y) =(u\with v)\wedge r(x,y).  
 \] 
 
 \emph{Case 2}: There is no idempotent element between $u$ and $v$. In this case, since $\with$ is a continuous t-norm on $[0,1]$, by Theorem \ref{Ordinal sum},  there is an open interval $(p,q)\subseteq[0,1]$ with both $p$ and $q$ being idempotent such that $u,v\in(p,q)$ and no point in $(p, q)$ is idempotent. Since $u\with u$ and $v\with v$ are both idempotent,  clearly one has that $p=u\with u=v\with v=u\with v$ and $p\with u=p\with v=p$.    On one hand, if   $r(x,y)\geq p$, then it holds that 
 $$\bigvee_{z\in X}(u\wedge r(x,z))\with (v\wedge r(z,y))\geq (u\wedge r(x,y))\with (v\wedge r(y,y))\geq p\with v=p=
 (u\with v)\wedge r(x,y).$$ 
  On the other hand, if $r(x,y)\leq p$, then it holds that 
  \[\bigvee_{z\in X}(u\wedge r(x,z))\with (v\wedge r(z,y))\geq (u\wedge r(x,y))\with (v\wedge r(y,y)) =r(x,y)\with v=r(x,y) = (u\with v)\wedge r(x,y).\]
 Therefore, the inequality 
    \[\bv_{z\in X}(u\with r(x,z))\wedge(v\with r(z,y))\geq (u\with v)\wedge r(x,y)\] 
    always holds as desired. 
\end{proof}

Collecting all   elements  $a$ in $[0,1]$ such that $a\with a$  is idempotent, then one obtains a complete sublattice 
 $$M=\{a\in[0,1]|a\with a\text{ is idempotent}\}$$
of $[0,1]$. Clearly, $(M,\with, 1)$ is  the largest      complete subquantale of $([0,1],\with,1)$ satisfying the forth condition in the above theorem. 
  
 \begin{thm}\label{stable ccc conq}
Let $\with$ be a continuous t-norm.  A stable subconstruct $\bfA$ of $\RCat$ is cartesian closed if and only if there is a suitable subset $S$ of $[0,1]^2$ contained in $M^2$ such that $\bfA=\RCat_S$. Hence ${M}\text{-}\mathbf{Cat}$ is the largest cartesian closed and stable subconstruct in $\RCat$.
\end{thm}
\begin{proof} It suffices to show the necessity. Suppose the cartesian closed and stable subconstruct $\bfA$ be the form of $\RCat_S$ for some suitable  subset $S$ of $[0,1]^2$.   For each $(p,q)\in S$, without  loss of generality,  let $p\leq q$. Then one has that both $(p,p)=(p,q)\wedge(q,p)\in S$ and $(q,q)=(p,q)\vee(q,p)\in S$. Therefore, the subset 
$$K=\{p\in[0,1]\mid(p,p)\in S\}$$
is a complete sublattice of $[0,1]$ and $(K,\with,1)$ is a complete subquantale of $([0,1],\&,1)$. Furthermore, it holds that 
$$K_\Delta\subseteq S\subseteq K^2.$$

    
Since $\RCat_{K_\Delta}$ is a stable subconstruct of $\RCat$ and it is also contained in $\RCat_{S}$,   it is a stable subcategory of $\RCat_S$. Hence, it is cartesian closed by Lemma \ref{stable of ccc}. By Theorem \ref{ccc},  it follows that $a\with a$ is idempotent for all $a\in K$. That means,   $S\subseteq M^2$ as desired.
\end{proof}
  
\begin{rem}\label{M} We describe the complete subquantale $(M,\with,1)$ explicitly in this remark. 
\begin{enumerate}
\item[(1)]  Let $\with$ be the G\"{o}del t-norm. Then $M=[0,1]$ since all elements in $[0,1]$ are idempotent. In this case, $\RCat$ itself is cartesian closed.

\item[(2)] Let $\with$ be the product t-norm. Then  $M=\{0,1\}$. In this case, $\Ord$ is the largest cartesian closed and stable subconstruct of $\RCat$. Additionally, through the isomorphism between $x\mapsto -\ln x$ and $([0,1],\cdot,1)$ to $([0,\infty]^{\mathrm{op}},+,0)$, the complete subquantale $(\{0,1\},\cdot,1)$ is transformed to the complete subquantale $(\{0,\infty\},+,0)$.  Therefore, the category $\mathbf{GMet}$ includes the isomorphic copy of  $\Ord$ as the largest cartesian closed and stable subconstruct. 
\item[(3)] Let $\with$ be the {\L}ukasiewicz t-norm. Then $M=[0,0.5]\cup\{1\}$. Notice that $\Lu_3=\{0,0.5,1\}\subseteq M$, thus, $M$-$\Cat$ is a cartesian closed subconstruct of $\RCat$ which is larger than  $\sLu_3$-$\Cat$. 
\item[(4)] Let $\with$ be a continuous t-norm in general. Collect   all Archimedean blocks $\{[a_i,b_i]\}_{i\in I}$ of $\&$ such that each one has an isomorphism $\varphi_i:([0,1],\&_\Lu,1)\lra([a_i,b_i],\&,b_i)$, then 
$$M=\{x\in[0,1]\mid x\text{ is idempotent}\}\cup\Big(\bigcup_{i\in I}[\varphi_i(0),\varphi_i(0.5)]\Big).$$
If the index set $I$ is not empty, the quantale $(M,\with,1)$ is not a frame since each $x$ in $(\varphi_i(0),\varphi_i(0.5)]$ is not idempotent, but the category $M$-$\Cat$ is cartesian closed. 

For example, define $\with:[0,1]^2\lra{[0,1]}$ by
\[x\with y=\begin{cases}
		2x\cdot y, \quad x,y\in[0,0.5],\\
		\max\{x+y-1, 0.5\},\quad x,y\in[0.5,1],\\
        x\wedge y, \quad \text{otherwise}.
		\end{cases} \]
Clearly $\with$ is a continuous t-norm. In this case, $\{0,0.5,1\}$ is the set of all idempotent elements of $\with$ and $M=[0.5,0.75]\cup\{0,1\}$.
\end{enumerate}  
\end{rem}



\section{Cartesian closed subcategories of $\mathbf{YCCat}$}
In this section, we show that, for each complete subquantale $(K,\with,1)$ of $(M,\with,1)$,   all Yoneda complete $K$-categories and Yoneda continuous $[0,1]$-functors form a cartesian closed category.

\begin{defn}
    A net $\xlam$ in a $[0,1]$-category $(X,r)$ is \emph{forward Cauchy} \cite{Wagner1994,Wagner1997} if $$\bv_{\lam\in D}\bw_{\lam\leq\mu\leq\ga}r(x_{\mu},x_{\ga})=1.$$
A forward Cauchy net converges to $a\in X$ if $$r(a,x)=\bv_{\lam\in D}\bw_{\lam\leq\mu}r(x_{\mu},x)$$ for all $x\in X$, and $a$ is called a \emph{Yoneda limit} of $\xlam$,  denoted as $a=\mathrm{lim} x_\lam.$
\end{defn}

A $[0,1]$-category $(X,r)$ is \emph{Yoneda complete} if every forward Cauchy net in $(X,r)$ converges. For a $[0,1]$-functor $f:(X,r)\lra (Y,s)$ and a forward Cauchy net $\xlam$ in $(X,r)$, $\{f(x_\lam)\}_{\lam\in D}$ is clearly a forward Cauchy net in $(Y,s)$, the $[0,1]$-functor $f$ is \emph{Yoneda continuous} if it preserves the Yoneda limit of all forward Cauchy nets in $(X,r)$, that is, $f(\mathrm{lim}x_\lam)=\mathrm{lim} f(x_\lam)$ for all forward Cauchy nets $\xlam$ in $(X,r)$.
All   Yoneda complete $[0,1]$-categories and Yoneda continuous $[0,1]$-functors constitute a category 
\[\YCom.\]

\begin{exmp}
    \begin{enumerate}
       \item[(1)] Every finite $[0,1]$-categories is trivially Yoneda complete.
        \item[(2)] When $K=\bbt=\{0,1\}$, a Yoneda complete $\bbt$-category is precisely a directed complete preordered set, and Yoneda continuous $[0,1]$-functors between Yoneda complete $\bbt$-categories  are exactly  Scott continuous maps. 
        \item [(3)] For all $x,y\in [0,1]$, define 
        $$d(x,y)=\sup\{z\in[0,1]\mid x\with z\leq y\},$$ 
        then $([0,1],d)$ is a Yoneda complete $[0,1]$-category \cite [Proposition 2.30]{Wagner1997}. 
    \end{enumerate}
\end{exmp}

Let $(X,r)$ and $(Y,s)$ be $[0,1]$-categories. All Yoneda continuous $[0,1]$-functors from $(X,r)$ to $(Y,s)$ constitute a  set  $[(X,r)\ra(Y,s)]$, which is a subset of of $[(X,r),(Y,s)]$. Equipped with the real-enriched relation $d_\otimes$ on  $([(X,r),(Y,s)]$, one obtains a $[0,1]$-category $([(X,r)\ra(Y,s)],d_\otimes)$.
 
\begin{prop} (\cite [Theorem 3.8]{Lai2016}) Let $(X,r)$ and $(Y,s)$ be both Yoneda complete  $[0,1]$-categories. Then both  
 $(X,r)\otimes(Y,s)$ and $([(X,r)\ra(Y,s)],d_\otimes)$ are  Yoneda complete. Therefore, the category $(\YCom,\otimes,\vmathbb{1})$ is monoidal closed.
 \end{prop}

 \begin{prop}
    Let $\bfA$ be stable subconstruct of $\RCat$. Then the category $(\bfA\cap\YCom,\otimes,\vmathbb{1})$ is monoidal closed.
\end{prop}
\begin{proof}
    For  given Yoneda complete $[0,1]$-categories $(X,r)$ and $(Y,s)$ in $\bfA$, the tensor product $(X,r)\otimes (Y,s)$ and the function space $([(X,r)\ra(Y,s)],d_{\otimes})$ in $\YCom$ are both in $\mathbf{A}$.
\end{proof}

\begin{prop}(\cite[Proposition 4.3]{Lai2016})\label{product}
Let $(X,r)$ and $(Y,s)$ be both Yoneda complete $[0,1]$-categories. Then the cartesian product $(X\times Y,r\times s)$ is also Yoneda complete. 
\end{prop}

But $\YCom$ need not be cartesian closed, since it is shown in \cite[Theorem 4.7]{Lai2016} that the category $\YCom$ is cartesian closed if and only if the continuous t-norm $\with=\wedge$.

For the complete subquantale $(M,\with,1)$ of $([0,1],\with,1)$, collect all the Yoneda complete $M$-categories and Yoneda continuous $[0,1]$-functors, one obtains a category \[\YComM.\]
We show that $\YComM$ is cartesian closed in the sequel. 

By Proposition \ref{powobj} and Theorem \ref{ccc}, the power object $(Y,s)^{(X,r)}$ of Yoneda complete $M$-categories  $(X,r)$ and $(Y,s)$  always exists in   $\MCat$. We check that its subobject $([(X,r)\ra(Y,s)],\dt)$ consisting of all Yoneda continuous $[0,1]$-functors from $(X,r)$ to $(Y,s)$   is also Yoneda complete.

\begin{prop}
Suppose $(X,r)$ and $(Y,s)$ are both Yoneda complete $M$-categories and $\flam$ is a forward Cauchy net in $([(X,r)\ra (Y,s)],\dt)$. Then $\flam$ is also forward Cauchy in $([(X,r)\ra (Y,s)],d_\otimes)$.
\end{prop}
\begin{proof}
Notice that $\dt(f,g)\leq d_\otimes(f,g)$ for all $f,g\in[(X,r)\ra(Y,s)]$.
\end{proof}

Therefore, if $\flam$ is a forward Cauchy net in $([X\ra Y],\dt)$, then for each $x\in X$, the net $\flamx$ is forward Cauchy in $(Y,s)$. So, one can define a function  
$$f:(X,r)\lra{(Y,s)}, \quad f(x)=\lim f_{\lam}(x)$$ in a pointwise way.  The function $f$  is a Yoneda continuous $[0,1]$-functor \cite [Theorem 4.2]{Lai2006}. Our aim is to show that $f$ is already a Yoneda limit of the net $\flam$ in $([X\ra Y],\dt)$, hence $([(X,r)\ra (Y,s)],\dt)$ is indeed Yoneda complete. 

\begin{defn}
 Let $(X,r)$ be a $[0,1]$-category and $\al$ be a real number in $[0,1]$. A net $\xlam$ is called to 
 be \emph{eventually $\al$-monotone} if there is some $\lam_0\in D$ such that $\al\leq r(x_\lam,x_\mu)$ for all $\mu\geq\lam\geq\lam_0$.
\end{defn}
\begin{lem}\label{al-mon}
    Let $(X,r)$ be a $[0,1]$-category and $\xlam$ be an eventually $\al$-monotone net in $(X,r)$ for some idempotent element $\alpha$ of $\with$.     Then there is some $\lam_0\in D$ such that for all $\mu\geq\lam\geq\lam_0$, 
    \[\al\wedge r(x_\mu,x)\leq \al\wedge r(x_\lam,x)\] for all $x\in X$.
\end{lem}

\begin{proof}
 Since $\xlam$ is eventually $\al$-monotone, there is some $\lam_0$ such that $\al\leq r(x_\lam, x_\mu)$ for all $\mu\geq\lam\geq\lam_0$. Thus, one has that 
  \[\al\wedge r(x_\mu,x)=\al\wedge\al\wedge r(x_\mu,x)=\al\wedge(\al\with r(x_\mu,x))
  \leq\al\wedge\big( r(x_{\lam},x_\mu)\with r(x_\mu,x)\big)\leq\al\wedge r(x_{\lam},x)\]
  for all   $\mu\geq\lam\geq\lam_0$ and all  $x\in X$.
\end{proof}

Recall from \cite{Gierz2003} that $x$ is \emph{way below } $y$ (written as $x\ll y$) in a directed complete poset $(L,\leq)$ if for all directed subset $D\subseteq{L}$,   $y\leq\bv D$ always implies the existence of a $d\in D$ with $x\leq d$. An element $x$ is \emph{compact} if $x\ll x$.
    
\begin{prop}\label{approx}
In the complete subquantale $(M,\with, 1)$ of  $([0,1],\with,1)$, denote the set of all idempotent elements of $\with$ by $\Idm$, it holds that 
 $$1=\bv\{\al\in \Idm\mid\al\ll 1 \text{ in } M\}.$$ 
\end{prop}

\begin{proof} By Theorem \ref{Ordinal sum} and Remark \ref{M}, we divide the proof  into two cases. 

\emph{Case 1}: There is some Archimedean block $[a,b]$ of the t-norm $\with$ satisfying that $b=1$. One can see that the top element $1\in M$ is an isolated point in $M$. Thus, one has that $1\ll 1$ in $M$ and   $1=\bv\{\al\in \Idm\mid\al\ll 1 \text{ in } M\}$.

\emph{Case 2}: Every Archimedean block $[a,b]$ of the t-norm $\with$ satisfies that $b<1$. One has that the top element $1\in M$ is not isolated in $\Idm$, hence, $1=\bv\{\al\in\Idm\mid\al<1\}$. Furthermore, for all idempotent $\al$ with $\al<1$, it holds that $\al\ll 1$ in $M$.  Thus, one obtains that $1=\bv\{\al\in \Idm\mid\al\ll 1 \text{ in } M\}$ as desired. 
    
\end{proof}

\begin{lem} \label{Cau is al-mon} Each forward Cauchy net $\xlam$ in an $M$-category $(X,r)$ is eventually $\al$-monotone for all   elements $\al\ll 1$ in the complete sublattice $M$.
\end{lem} 
\begin{proof} Notice that 
$$\{\al_\lam=\bw_{\lam\leq\mu\leq\ga}r(x_\mu,x_\ga)\mid \lam\in D\}$$
is a directed subset in $M$ and $\bv_{\lam\in D}\al_\lam=1$. If $\al\ll 1$, then there is some $\lam_0\in D$ such that
$$\al\leq\al_{\lam_0}=\bw_{\lam_0\leq\mu\leq\ga}r(x_\mu,x_\ga).$$
Thus, for all $\ga\geq\mu\geq\lam_0$, $\al\leq r(x_\mu,x_\ga) $, which means that $\xlam$ is eventually $\al$-monotone. 
\end{proof}


\begin{thm}\label{exp in Kcat}
Suppose $(X,r)$ and $(Y,s)$ are  Yoneda complete $M$-categories, then the $M$-category $([(X,r)\ra(Y,s)],\dt)$ is also Yoneda complete.
\end{thm}

\begin{proof} Suppose $\flam$ is a forward Cauchy net in $([(X,r)\ra(Y,s)],\dt)$. We claim that the Yoneda continuous $[0,1]$-functor $f:(X,r)\lra{(Y,s)}$ given by 
$$f(x)=\lim f_{\lam}(x)$$ is a Yoneda limit of $\flam$ in $([(X,r)\ra(Y,s)],\dt)$. That is, we need to check that, for all $g\in [(X,r)\ra (Y,s)]$,   
    \[\dt(f,g)=\bv_{\lam\in D}\bw_{\lam\leq\mu}\dt(f_\mu,g).\]

 Let $\al$ be an idempotent element such that $\al\ll 1$ in $M$.  By Lemma \ref{Cau is al-mon},  there is a $\lam_{0}\in D$ such that for all $\mu\geq\lam\geq\lam_0$, 
 $$\al\leq\dt(f_\lam,f_\mu)\leq s(f_\lam(x),f_\mu(x))$$ for all $x\in X$.
Furthermore, by Lemma \ref{al-mon}, it holds that, for all $\mu\geq\lam\geq\lam_0$, 
$$\al\wedge\dt(f_\mu,g)\leq\al\wedge\dt(f_\lam,g)$$  and
$$\al\wedge s(f_\mu(x),g(y))\leq\al\wedge s(f_\lam(x),g(y))$$ for all  $g\in[(X,r)\ra(Y,s)]$ and $x,y\in X$.

Thus, it follows that
\begin{align*}
\al\wedge\bv_{\lam\in D}\bw_{\lam\leq\mu}\dt(f_\mu,g)
&=\bv_{\lam\in D}\bw_{\lam\leq\mu}\al\wedge\dt(f_\mu,g)\\
&=\bv_{\lam_0\leq\lam}\bw_{\lam\leq\mu}\al\wedge\dt(f_\mu,g)\\
&=\bw_{\lam_0\leq\mu}\al\wedge\dt(f_\mu,g)\\
&=\al\wedge\bw_{\lam_0\leq\mu} \dt(f_\mu,g),
\end{align*}
and similarly,   
\begin{align*}\al\wedge\bv_{\lam\in D}\bw_{\lam\leq\mu} s(f_\mu(x),g(y))
&=\bv_{\lam\in D}\bw_{\lam\leq\mu}\al\wedge\dt(f_\mu(x),g(y))\\
&=\bv_{\lam_0\leq\lam}\bw_{\lam\leq\mu}\al\wedge\dt(f_\mu(x),g(y))\\
&=\bw_{\lam_0\leq\mu}\al\wedge\dt(f_\mu(x),g(y))\\
&=\al\wedge\bw_{\lam_0\leq\mu}  s(f_\mu(x),g(y)).
\end{align*} for all $g\in[(X,r)\ra(Y,s)]$ and $x,y\in X$.

Therefore, for all $g\in [(X,r)\ra (Y,s)]$, one can calculate as below:
\begin{align*}
\al\wedge\dt(f,g)&=\al\wedge\left(\bw_{x,y\in X}r(x,y)\ra s(f(x),g(y))\right)\\
        &=\left(\bw_{x,y\in X}r(x,y)\ra a)\right)\wedge\left(\bw_{x,y\in X}r(x,y)\ra s(f(x),g(y))\right)\\
        &=\bw_{x,y\in X}\Big(\big(r(x,y)\ra a\big)\wedge\big(r(x,y)\ra s(f(x),g(y))\big)\Big)\\
        &=\bw_{x,y\in X}r(x,y)\ra(\al\wedge s(f(x),g(y)))\quad\quad\quad\quad\quad (\text{Proposition \ref{resd prop} } (3))\\
        &=\bw_{x,y\in X}r(x,y)\ra \Big(\al\wedge\bv_{\lam\in D}\bw_{\lam\leq\mu}s(f_\mu(x),g(y))\Big)\\
        &=\bw_{x,y\in X}r(x,y)\ra \Big(\al\wedge\bw_{\lam_0\leq\mu}s(f_\mu(x),g(y))\Big)\\
        &=\bw_{x,y\in X}\Big(r(x,y)\ra \al\Big)\wedge\Big(r(x,y)\ra \bw_{\lam_0\leq\mu}s(f_\mu(x),g(y))\Big)\quad\quad (\text{Proposition \ref{resd prop} }(3))\\
        &=\Big(\bw_{x,y\in X}r(x,y)\ra \al\Big)\wedge\Big(\bw_{x,y\in X}r(x,y)\ra \bw_{\lam_0\leq\mu}s(f_\mu(x),g(y))\Big)\\
        &=\al\wedge\bw_{\lam_0\leq\mu}\bw_{x,y\in X}r(x,y)\ra s(f_\mu(x),g(y))\quad\quad (\text{Proposition \ref{resd prop} }(3))\\
        &=\al\wedge\bw_{\lam_0\leq\mu}\dt(f_\mu,g)\\
        &=\al\wedge\bv_{\lam\in D}\bw_{\lam\leq\mu}\dt(f_\mu,g). 
    \end{align*}
    
Since the complete subquantale $(M,\with,1)$ of $([0,1],\with,1)$ satisfies the property $$1=\bv\{\al\in \Idm\mid a\ll 1 \text{ in } M\},$$   it follows that    
 \[\dt(f,g)=\bv_{{\al\in \Idm\atop \al\ll 1} }\Big(\al \wedge \dt(f,g)\Big)=\bv_{{\al\in \Idm\atop \al\ll 1} }\left(\al\wedge  \bv_{\lam\in D}\bw_{\lam\leq\mu}\dt(f_\mu,g)\right)=\bv_{\lam\in D}\bw_{\lam\leq\mu}\dt(f_\mu,g).\] for all $g\in [(X,r)\ra (Y,s)]$.     
Hence, $f$ is a Yoneda limit of $\flam$ as desired.
\end{proof}

\begin{prop}\label{separate}
     Let $(X,r),(Y,s),(Z,t)$ be Yoneda complete $[0,1]$-categories. A $[0,1]$-functor $f:(X,r)\times (Y,s)\lra (Z,t)$ is Yoneda continuous if and only if it is Yoneda continuous separately.
\end{prop}
\begin{proof}
    The proof is similar to that of \cite[Proposition 3.7]{Lai2016}.
\end{proof}

\begin{prop}\label{evmap} 
 Suppose $(X,r)$ and $(Y,s)$ are  Yoneda complete $M$-categories, then the evaluation map $\ev:(X,r)\times{([(X,r)\ra(Y,s)],\dt)}\lra{(Y,s)}$ given by $\ev(x,f)=f(x)$ is a Yoneda continuous $[0,1]$-functor.
\end{prop}
\begin{proof}
    It is clear that $\ev$ is a $[0,1]$-functor and $\ev(-,f):(X,r)\lra{(Y,s)}$  is Yoneda continuous for each $f\in[(X,r)\ra(Y,s)]$. Given an $x\in X$  and a forward Cauchy net $\flam$ in $([(X,r)\ra(Y,s)],\dt)$, by Theorem \ref{exp in Kcat}, the function $f:X\lra{Y}$ given by $f(x)=\lim f_{\lam}(x)$ is a Yoneda limit of $\flam$.  It follows that $\ev(x,f)=f(x)=\lim f_{\lam}(x)$, hence $\ev(x,-)$ is Yoneda continuous. By Proposition \ref{separate}, $\ev$ is Yoneda continuous.
\end{proof}

\begin{prop}\label{transp}  Suppose $(X,r)$, $(Y,s)$ and $(Z,t)$ are all Yoneda complete $M$-categories and $f:(X,r)\times{(Z,t)}\lra{(Y,s)}$ is a Yoneda continuous $[0,1]$-functor, then $\hat{f}:(Z,t)\lra{([(X,r)\ra(Y,s)],\dt)}$ is also a Yoneda continuous $[0,1]$-functor, where $\hat{f}(z)=f(-,z)$ for each $z\in Z$.
\end{prop}
\begin{proof}
    It is easily verified that $\hat{f}$ is a $[0,1]$-functor. Given a forward Cauchy net $\zlam$ in $(Z,t)$ which converges to $a\in Z$, then $\{f(-,z_\lam)\}_{\lam\in D}$ is a forward Cauchy net in $([(X,r)\ra(Y,s)],\dt)$ with a pointwise   Yoneda limit $g$, which is given by 
    $$g(x)=\lim f(x,z_\lam).$$ 
 Since $f$ is Yoneda continuous, it follows that $\hat{f}(a)(x)=f(x,a)=\lim f(x,z_\lam)=g(x)$ for all $x\in X$, that is,  $\hat{f}(a)=g$. The $[0,1]$-functor $\hat{f}$ is Yoneda continuous as desired.
\end{proof}

Combine the results of Proposition \ref{product}, Theorem \ref{exp in Kcat}, Proposition \ref{evmap} and Proposition \ref{transp}, we have the following:
\begin{thm}\label{YocomK is ccc}
      The category $\YComM$ is cartesian closed.
\end{thm}




\begin{cor}
Let $(K,\with,1)$ be a complete subquantale of $(M,\with,1)$. 
\begin{itemize}
\item[(1)] The category ${K}\text{-}\mathbf{YCCat}$ is cartesian closed.
\item[(2)] All  symmetric $K$-categories in $K$-$\YCom$ and Yoneda continuous $[0,1]$-functors form a cartesian closed category. 
\end{itemize}
\end{cor}
\begin{proof}(1): Notice that, for Yoneda complete $K$-categories $(X,r)$ and $(Y,s)$, both the product $(X,r)\times (Y,s)$ and the power object $([(X,r)\ra(Y,s)],\dt)$ in $\YComM$ are also $K$-categories. 

(2): For symmetric Yoneda complete $K$-categories $(X,r)$ and $(Y,s)$, it is easy to see that product $(X,r)\times (Y,s)$ is also a symmetric Yoneda complete $K$-category. It suffices to check that the power object $([(X,r)\ra(Y,s)],\dt)$ in $\YComM$ is a symmetric $K$-category, that is, for all $f,g\in [(X,r)\ra(Y,s)]$:
	\begin{align*}
		\dt(f,g)&=\bigwedge_{x,y\in X}r(x,y)\ra s(f(x),g(y))\\
		&=\bigwedge_{x,y\in X}r(y,x)\ra s(g(y),f(x))\\
		&=\dt(g,f).
	\end{align*}
\end{proof}

\begin{exmp} Consider the {\L}ukasiewicz norm $\with_{\L}$. Notice that $M=[0,0.5]\cup\{1\}$ and ${\L}_3=\{0,0.5,1\}\subseteq M$. Clearly, $({\L}_3,\&_{\L},1)$ is a complete subquantale of $(M,\&_{\L},1)$. Therefore, both $\YComM$ and ${\L}_3$-$\mathbf{YCCat}$ are cartesian closed. The fact that category ${\L}_3$-$\mathbf{YCCat}$ is cartesian closed is firstly shown in \cite{Liu2017}. 
\end{exmp}

If $(K,\with,1)$ is a complete subquantale of $(M,\with,1)$, then both $K^2$ and $K_\Delta$ are suitable subsets of $[0,1]^2$ contained in $M^2$. By the above Corollary, $\YCom\cap\RCat_{K^2}$ and $\YCom\cap\RCat_{K_\Delta}$ are cartesian closed. However, it remains open whether $\YCom\cap\mathbf{A}$ is cartesian closed for all cartesian closed and stable subconstruct $\mathbf{A}$ of $\RCat$.  

\section{Conclusion}

Let $\with$ be a continuous triangular norm with Archimedean blocks, that is, $\&$ is not idempotent. Neither the category $\RCat$ nor the category $\YCom$  is  cartesian closed. Some cartesian closed subcategories of them are found out.  Firstly,  all cartesian closed and stable subconstructs of   $\RCat$ is characterized by   suitable subsets of $[0,1]^2$. It is shown that a stable subconstructs of $\RCat$ is cartesian closed if and only if the associated suitable subset $S\subseteq M^2$, and then $\MCat$ is the largest one among all cartesian and stable subconstructs.   Secondly, it is shown that $\YComK$ is cartesian closed if $(K,\with,1)$ is a complete subquantale of $(M,\with,1)$. It is remarkable that the  quantale $(M,\with,1)$ need not be a frame but both the category $\MCat$ and $\YComM$ are cartesian closed.

 \section*{Declaration of competing interest}
The authors declare that they have no known competing financial interests or personal relationships that could have appeared to influence the work reported in this paper.
\section*{Acknowledgment}
The authors  express their sincere thanks to the referees for their most valuable comments and suggestions on this paper and acknowledge the support of National Natural Science Foundation of China (12171342).


\begin{thebibliography}{[99]}\setlength{\itemsep}{0pt}

\bibitem{Acc} J. Ad\'{a}mek, H. Herrlich, G.E. Strecker, \emph{Abstract and Concrete Categories}, Wiley, New York, 1990.

\bibitem{America1989} P. America, J.J.M.M. Rutten, Solving reflexive domain equations in a category of complete metric spaces, Journal of Computer and System Sciences 39 (1989) 343--375.

\bibitem{Antoniuk2011} S. Antoniuk, P. Waszkiewicz, A duality of generalized metric spaces, Topology and its Applications 158 (2011) 2371--2381.


\bibitem{Birk1973} G. Birkhoff, \emph{Lattice Theory}, Third Edition, American Mathematical Society, 1973.

\bibitem{Belo2002}R. B\v{e}lohl\'{a}vek, \emph{Fuzzy Relational Systems: Foundations and Principles}, Kluwer Academic Publishers, Dordrecht, 2002.
\bibitem{BBR1998} M. Bonsangue, F. Breugel, J.J.M.M. Rutten, Generalized metric spaces: Completion, topology, and powerdomains via the Yoneda embedding,
Theoretical Computer Science 193 (1998) 1--51.


\bibitem{Clem03} M.M. Clementino, D. Hofmann, W. Tholen, Exponentiability in categories of lax algebras, Theory and Applications of Categories 11 (2003) 337--352.
\bibitem{Clem06} M.M. Clementino, D. Hofmann, Exponentiation in $\mathbf{V}$-categories, Topology and its Applications 153 (2006) 3113--3128.


\bibitem{Tholen1987} R. Dyckhoff, W. Tholen, Exponentiable morphisms, partial products and pullback complements, Journal of Pure and Applied Algebra 49 (1987) 103--116.

\bibitem{Fan2001} L. Fan, A new approach to quantitative domain theory, Electronic Notes in Theoretical Computer Science 45 (2001) 77--87.
 

 \bibitem{Flagg1997} B. Flagg, R. Kopperman, Continuity spaces: Reconciling domains and metric spaces, Theoretical Computer Science 177 (1997) 111--138.
\bibitem{Flagg2002} R.C. Flagg, Ph. S\"{u}nderhauf, The essence of ideal completion in quantitative form, Theoretical Computer Science 278 (2002) 141--158.
 
\bibitem{Gierz2003} G. Gierz, K.H. Hofmann, K. Keimel, J.D. Lawson, M. Mislove,   D.S. Scott, \emph{Continuous Lattices and Domains},   
Cambridge University Press, Cambridge, 2003.


\bibitem{Gutierres2013} G. Gutierres, D. Hofmann, Approaching metric domains,  Applied Categorical Structures 21 (2013) 617--650.

\bibitem{Hajek1998} P. H\'{a}jek, \emph{Metamathematics of Fuzzy Logic}, Kluwer Academic Publishers, 1998.

\bibitem{Hofmann2011} D. Hofmann, P. Waszkiewicz, Approximation in quantale-enriched categories, Topology and its Applications 158 (2011) 963--977.

\bibitem{Hofmann2018} D. Hofmann, P. Nora, Enriched Stone-type dualities, Advances in Mathematics
 330 (2018)  307--360.


\bibitem{Klement2000} E.P. Klement, R. Mesiar, E. Pap,  \emph{Triangular Norms}, 
	Kluwer Academic Publishers, Dordrecht, 2000.
 
\bibitem{Kunzi2002} H.P. Künzi, M.P. Schellekens, On the Yoneda completion of a quasi-metric space, Theoretical
Computer Science 278 (2002) 159--194.

 \bibitem{Lai2006} H. Lai, D. Zhang, Complete and directed complete $\Omega$-categories, Theoretical Computer Science 388 (2007) 1--25.

\bibitem{Lai2016} H. Lai, D. Zhang, Closedness of the category of liminf complete fuzzy orders, Fuzzy Sets and Systems 282 (2016) 86--98.
 
	\bibitem{Lawvere1973}  F.W. Lawvere, Metric spaces, generalized logic, and closed categories, Rendiconti del Seminario Mat\'{e}matico e Fisico di Milano 43 (1973) 135--166.



 \bibitem{Liu2017} M. Liu, B. Zhao, A non-frame valued cartesian closed category of liminf complete fuzzy orders, 
Fuzzy Sets and Systems 321 (2017) 50--54.
 

 
\bibitem{Lowen2015} R. Lowen, \emph{Index Analysis, Approach Theory at Work}, Springer, 2015.


\bibitem{Rosen1990} K.I. Rosenthal, \emph{Quantales and Their Applications}, Longman, Essex, 1990.

 
 \bibitem{Wagner1994} K.R. Wagner, Solving recursive domain equations with enriched categories, Ph.D. Thesis, Carnegie Mellon University, Technical report CMU-CS-94-159, July 1994.
 
\bibitem{Wagner1997} K.R. Wagner, Liminf convergence in $\Omega$-categories, Theoretical Computer Science 184 (1997) 61--104.



\bibitem{Zade1971} L.A. Zadeh, Similarity relations and fuzzy orderings, Information Sciences 3 (1971) 177--200.

\bibitem{Zhang2024} D. Zhang, Introductory notes on real-enriched categories, 	arXiv:2403.09716 [math.CT].
\end{thebibliography}
\end{document}